\documentclass[12pt,a4paper]{article}
\usepackage{amsfonts}
\usepackage{mathrsfs}
\usepackage{amsmath}
\usepackage{color}
\usepackage{stmaryrd}
\usepackage{amssymb}
\usepackage{bbm}
\usepackage{graphicx}
\usepackage{enumerate}
\usepackage{theorem}
\usepackage{ulem}
\usepackage{authblk}
\makeatletter
\def\tank#1{\protected@xdef\@thanks{\@thanks
        \protect\footnotetext[0]{#1}}}
\def\bigfoot{

    \@footnotetext}
\makeatother

\newcommand{\ea}{\end{array}}
\newtheorem{theorem}{Theorem}[section]
\newtheorem{proposition}{Proposition}[section]

\newtheorem{lemma}{Lemma}[section]
\newtheorem{definition}{Definition}[section]
\newtheorem{remark}{Remark}[section]

{\theorembodyfont{\rmfamily}
}

\newenvironment{proof}{Proof.}

\allowdisplaybreaks

\begin{document}
\title {\Large \bf Exponential ergodicity of Stochastic Evolution Equations with reflection} 
\author[1]{Zdzisław Brzeźniak\thanks{ E-mail:zdzislaw.brzezniak@york.ac.uk}}
\author[2]{Qi Li\thanks{ E-mail:vivien777@mail.ustc.edu.cn}}
\author[3]{Tusheng Zhang\thanks{ E-mail:Tusheng.Zhang@manchester.ac.uk}}

\affil[1]{Department of Mathematics, University of York, Heslington, YO10 5DD, York, United Kingdom.}
\affil[2]{School of Mathematical Sciences, University of Science and Technology of China, Hefei, 230026, People's Republic of China.}
\affil[3]{Department of Mathematics, University of Manchester, Oxford Road,
Manchester, M13 9PL, UK.}
\renewcommand\Authands{ and }
\date{}
\maketitle

\begin{center}
\begin{minipage}{130mm}
{\bf Abstract.}
In this paper, we establish an exponential ergodicity for stochastic evolution equations with reflection in an infinite dimensional ball. As an application, we obtain the exponential ergodicity of stochastic Navier-Stokes equations with reflection. A coupling method plays an important role.

\vspace{3mm} {\bf Keywords.} stochastic evolution equations; reflection; invariant measure; ergodicity; coupling method.

\end{minipage}
\end{center}

\section{Introduction}
\setcounter{equation}{0}
 \setcounter{definition}{0}
Let $H$ be a seperable Hilbert space with the norm $|\cdot|_H$ and inner product $(\cdot,\cdot)$. $D:=B(0,1)$ denotes the open unit ball on $H$. Let $A$ be a self-adjoint, positive definite  operator on $H$ and let $B$ be a unbounded bilinear map from $H\times H$ to $H$. Let $l^2$ denote the Hilbert space of all sequences of square summable real numbers with standard norm $\|\cdot\|_{l^2}$. Consider the following stochastic evolution equations(SEEs) with reflection:
\begin{equation}\label{SEE}
\left\{
\begin{aligned}
dX(t)&+AX(t)dt =f(X(t))dt +B(X(t),X(t))dt + \sigma(X(t))dW(t)+dL(t), t\geq0,\\
 X(0)& =x_0, \quad x_0\in \overline{D},
\end{aligned}
\right.
\end{equation}
Where $X$ is a $\overline{D}$-valued continuous stochastic process and $L$ is a $H$-valued continuous stochastic process, which plays the role of a local time. $\overline{D}$ denotes the closed unit ball in the Hilbert space $H$. $f$ and $\sigma$ are measurable mappings. And $\sigma(X(t))dW(t)= \sum_{i=1}^\infty \sigma_i(X(t))\beta_i(t)$, where $\{\beta_i,i\geq 0\}$ are real-valued standard (scalar) mutually independent Wiener processes on a filtered probability space $(\Omega,\mathcal{F,\mathbb{F}}=(\mathcal{F}_t)_{t\geq 0},P)$ satisfying the usual conditions. The entry of $\sigma$ are given as follows:
$$\sigma:H\rightarrow H\times l^2, \quad \sigma_i(u):H\rightarrow H.$$
Indeed, $\sigma(u),u\in H$ can be regarded as a Hilbert-Schmidt operator from $l^2$ to $H$ whose norm is denoted by $L_2(l^2,H)$.

Stochastic partial differential equations(SPDEs) with reflection can be used to model the evolution of random interfaces near a hard wall, see \cite{FO}. Existence and uniqueness of the above stochastic reflected problems were established in \cite{BZ}. For the study of real-valued SPDEs with reflection we refer the readers to  \cite{NP}, \cite{DP}, \cite{XZ} and references therein.

In this paper, we are concerned with the exponential ergodicity of equation (\ref{SEE}). Ergodicity for SEEs and stochastic partial differential equations(SPDEs) has been studied extensively. For instance, with the irreducibility in hand, we can obtain the ergodicity by proving the strong Feller property(see \cite{DZ,PZ,Z}), or the asymptotic strong Feller(see \cite{HM}), or the e-property(see \cite{KPS,KSS}).

 Unfortunately, verifying these conditions for SEEs with reflection turns out to be quite challenging. To overcome these difficulties, we adopt a coupling approach for exponential or subexponential ergodicity as proposed in \cite{HMS}. The method is based on the concept of a $d$-small set, which is a much more general object than the so called the small set; see Section 4.1 below for the precise definition.  In a sense, it shows that if a Markov process visits a $d$-small set frequently enough, then, under certain additional constraints on a Markov kernel(the so-called nonexpansion property), there exists a unique invariant probability measure. Moreover, the transition probabilities weakly converge to the invariant probability measure with a rate  quantified by a suitably chosen probability metric and the convergence rate is determined by the recurrence properties of the process. Then, the crucial step of the current work is to construct a \textit{distance-like} function $d$ such that the Markov kernel associated with the SEEs with reflection indeed has the nonexpansion property w.r.t. $d$ and that a certain set is indeed $d$-small.

 The rest of the paper is organized as follows. In Section 2, we introduce the reflected stochastic evolution equations and the precise framework. Section 3 is devoted to the study the Feller property of the solutions to Eq. (\ref{SEE}) and the existence of invariant measures for the Feller semigroup. In Section 4.1, we recall a sufficient condition  for the exponential ergodicity. The main results are presented in Section 4.2.

\section{Framework}
\setcounter{equation}{0}
 \setcounter{definition}{0}

Let $A$ be a self-adjoint, positive definite operator on the Hilbert space $H$ such that there exists $\lambda_1>0$ satisfying
\begin{equation}
(Au,u)\geq\lambda_1|u|_H^2, \quad u\in D(A).
\end{equation}

Set $V:=D(A^{\frac{1}{2}})$, the domain of the operator $A^{\frac{1}{2}}$. Then V is a Hilbert space with the inner product
\begin{equation}
((u,v))=(A^{\frac{1}{2}}u,A^{\frac{1}{2}}v), \quad u,v\in V,
\end{equation}
and the norm$\|\cdot\|$. By $V^{\ast}$ we denote the dual space of $V$, so that we have a Gelfand triple
\begin{equation}
V\hookrightarrow H\cong H^{\ast} \hookrightarrow V^{\ast}.
\end{equation}
We also assume that $V$ is  compactly embedded into $H$. And we use $\langle\cdot,\cdot\rangle$ to denote the duality between $V$ and $V^{\ast}$.\\

The following hypothesis will be in force throughout this work.\\
\vskip 0.3cm
\textbf{(A.1)} Let $f:H\rightarrow V^{\ast}$ and $\sigma:H\rightarrow L_2(l^2,H)$ be two measurable maps such that there exists a constant $C_1$ satisfying
\begin{equation}\label{A1}
|f(u)-f(v)|^2_{V^{\ast}} + |\sigma(u)-\sigma(v)|^2_{ L_2(l^2,H)} \leq C_1|u-v|^2_H, \quad \text{for all }u,v\in H.
\end{equation}
\textbf{(A.2)} Consider a bilinear map $B:V\times V \rightarrow V^{\ast}$ and the corresponding trilinear form $\overline{b}:V\times V\times V\rightarrow \mathbb{R}$ defined by
\begin{equation}
\overline{b}(u,v,w)=\langle B(u,v),w\rangle, \quad u,v,w\in V.
\end{equation}
Assume that the form $\overline{b}$ satisfies the following conditions.\\
a) For all $u,v,w\in V$,
\begin{equation}
\langle B(u,v),w\rangle= \overline{b}(u,v,w)= -\overline{b}(u,w,v)= -\langle B(u,w),v\rangle.
\end{equation}
b) For all $u,v,w\in V$,
\begin{equation}
|\langle B(u,v),w\rangle| = |\overline{b}(u,v,w)|\leq 2 \|u\|^{\frac{1}{2}} |u|_H^{\frac{1}{2}} \|w\|^{\frac{1}{2}}|w|_H^{\frac{1}{2}}\|v\|.
\end{equation}
\vspace{1em}
\textbf{(A.3)} $u_0\in \overline{D}$ is deterministic.\\
Assumption (A.2) particularly implies that
\begin{equation}
\begin{aligned}
&\overline{b}(u,v,v)=0 \quad i.e.\quad\langle B(u,v),v\rangle=0, \quad u,v\in V,\\
&\|B(u,u)\|_{V^{\ast}}\leq 2\|u\||u|_H, \quad u\in V.
\end{aligned}
\end{equation}
Throughout the paper, $C$ will denote a generic constant whose value may be different from line to line.\\
Now we recall the definition of a solution given in \cite{BZ}.
\begin{definition}\label{defsol}
A pair $(X,L)$ is said to be a solution of the reflected problem $(\ref{SEE})$ iff the following conditions are satisfied\\
(i) $X$ is a $\overline{D}$-valued continuous and $\mathbb{F}$-progressively measurable stochastic process with $X\in L^2([0,T];V)$, for any $T>0$, $\mathbb{P}$-a.s.\\
(ii) the corresponding $V$-valued process is strongly $\mathbb{F}$-progressively measurable;\\
(iii) $L$ is $H$-valued, $\mathbb{F}$-progressively measurable stochastic process of paths of locally bounded  variation such that $L(0)=0$ and
\begin{equation}
\mathbb{E}[|\text{Var}_H(L)([0,T])|^2]<+ \infty,\quad T\geq 0,
\end{equation}
where, for a function $v:[0,\infty)\rightarrow H$,Var$_H(v)([0,T])$ is the total variation of $v$ on $[0,T]$ defined by
\begin{equation}
\text{Var}_H(L)([0,T]):= sup\sum_{i=1}^n|v(t_i)-v(t_{i-1})|_H,
\end{equation}
where the supremum is taken over all partitions $0 = t_0 < t_1 <\cdots < t_{n-1} < t_n = T$, $n\in \mathbb{N}$, of the interval $[0,T]$;\\
(iv) $(X,L)$ satisfies the following integral equation in $V^{\ast}$, for every $t\geq 0$, $\mathbb{P}$-almost surely,
\begin{equation}\nonumber
X(t)+ \int_0^t AX(s)ds -\int_0^t f(X(s))ds - \int_0^t B(X(s),X(s))ds= u_0 +\int_0^t\sigma(X(s))dW(s) +L(t);\\
\end{equation}
(v) for every $T>0$, and $\phi \in C([0,T],\overline{D})$, $\mathbb{P}$-almost surely,
\begin{equation} \label{def5}
\int_0^T (\phi(t)-X(t),L(dt))\geq 0.
\end{equation}
where the integral on LHS is Riemann-Stieltjes integral of the $H$- valued function $\phi-X$ with respect to an $H$-valued bounded-variation function $L$.
\end{definition}
Let us recall the following result from \cite{BZ}.
\begin{proposition} Let the assumptions \textbf{(A.1)-(A.3)} hold. The reflected stochastic evolution equation(\ref{SEE}) admits a unique solution $(X,L)$ that satisfies, for $T>0$,
\begin{equation} \label{esti}
\mathbb{E}[\sup\limits_{t\in[0,T]}|X(t)|_H^2 + \int_0^T\|X(s)\|^2ds]<\infty.
\end{equation}
\end{proposition}

\section{Feller property and invariant measures}
\setcounter{equation}{0}
 \setcounter{definition}{0}
For fixed initial value $x_0= v\in H$, we denote the unique solution of equation (\ref{SEE}) by $(X^v(t),L^v(t))$. Then $\{X^v(t):v\in \overline{D}, t\geq 0\}$ forms a strong Markov process with state space $H$.\\
Let $C_b(H)$ denote the set of all bounded continuous functions on $H$. Then $C_b(H)$ is clearly a Banach space under the sup norm
$$\|\phi\|_\infty:=\underset{u\in H}{sup}\ |\phi(u)|.$$
For $t>0$, we define the semigroup $\mathbf{T_t}$ associated with $\{X^v(t):v\in \overline{D}, t\geq 0\}$ by
$$ \mathbf{T_t}\phi(v):=\mathbb{E}(\phi(X^v(t))),\quad \phi \in C_b(H).$$
We have:
\begin{theorem}
Under \textbf{(A.1)}-\textbf{(A.3)}, for every $t>0$, $\mathbf{T_t}$ maps $C_b(H)$ into $C_b(H)$. That is, $(\mathbf{T_t})_{t\geq0}$ is a Feller semigroup on $C_b(H)$.
\end{theorem}
\begin{proof}
Let $\phi$ in $C_b(H)$ be given. Let $h(t)=exp(-4\int_0^t\|X^v(s)\|^2ds)$ and we claim that
\begin{equation}\label{3.1}
\mathbb{E}[\underset{0\leq s\leq t}{sup}\ h(s)\left|X^v(s)-X^{v'}(s)\right|_H^2]\leq C_t |v-v'|_H^2,
\end{equation}
 where $C_t$ is a constant depending on $t$ .\vspace{1em}\\
Write $(X(t),L(t))=(X^v(t),L^v(t)), (X'(t),L'(t))=(X^{v'}(t),L^{v'}(t))$ and
$$ w(t)=X(t)-X'(t).$$
By Ito's formula, we have
\begin{align}\label{3.2}
\begin{aligned}
h(t)|w(t)|_H^2 +& 2\int_0^t h(s)(w(s), Aw(s))ds\\
=&|v-v'|^2-4\int_0^t h(s)\|X(s)\|^2|w(s)|_H^2 ds\\
 & +2\sum_{i=1}^\infty\int_0^t h(s)\langle w(s), \sigma_i(X(s)) -\sigma_i(X'(s))\rangle d\beta_i(s)\\
 & +2\int_0^t h(s)\langle w(s), f(X(s)) -f(X'(s))\rangle ds\\
 & +2\int_0^t h(s)\langle w(s), B(X(s)) -B(X'(s))\rangle ds\\
 & +2\int_0^t h(s)\langle w(s), L(ds) -L'(ds)\rangle\\
 & +\int_0^t h(s)|\sigma(X(s))-\sigma(X'(s))|_{L_2(l^2,H)}^2ds,\\
\end{aligned}
\end{align}
where $B(X(s))=B(X(s),X(s))$. As $X(t),X'(t)\in \overline{D}$, for all $t\geq 0$, we infer from the definition of the solution that
$$\int_0^t h(s)\langle w(s), L(ds) -L'(ds)\rangle\leq 0.$$
By assumptions on $f$ and $B$ and following the similar arguments as in \cite{BZ} Lemma 4.3, we obtain
\begin{align}\label{3.3}
\begin{aligned}
&2\int_0^t h(s)\langle w(s), f(X(s)) -f(X'(s))\rangle ds\\
\leq & C \int_0^t h(s)\| w(s)\| | f(X(s)) -f(X'(s))|_{V^\ast} ds\\
\leq & C \int_0^t h(s)\| w(s)\| |w(s)|_H ds\\
\leq & \frac{1}{2} \int_0^t h(s)\| w(s)\|^2ds + C \int_0^t h(s)| w(s)|_H^2 ds,
\end{aligned}
\end{align}
and
\begin{align}\label{3.4}
\begin{aligned}
&2\int_0^t h(s)\langle w(s), B(X(s)) -B(X'(s))\rangle ds\\
\leq & 2\int_0^t h(s)|\overline{b}(w(s),X(s),w(s)| ds\\
\leq & 4\int_0^t h(s)|w(s)|_H\|X(s)\|\|w(s)\| ds\\
\leq & \int_0^t h(s)\|w(s)\| ^2ds + 4\int_0^t h(s)|w(s)|_H^2\|X(s)\|^2 ds.\\
\end{aligned}
\end{align}
Substituting (\ref{3.4}), (\ref{3.3}) into (\ref{3.2}),  using the Burkholder inequality as well as the Lipschitz continuity of maps $\sigma$, we get that
$$
\mathbb{E}\left[\underset{0\leq s\leq t}{sup}\ h(s)\left|w(s)\right|_H^2\right] + \mathbb{E}[\int_0^t h(s)\|w(s)\|^2ds]\leq |v-v'|^2 + C\mathbb{E}[\int_0^t h(s)| w(s)|_H^2 ds].
$$
By Gronwall Lemma, we deduce that
\begin{equation}\label{3.5}
\mathbb{E}\left[\underset{0\leq s\leq t}{sup}\ h(s)\left|X^v(s)-X^{v'}(s)\right|_H^2\right]\leq C_t |v-v'|_H^2.
\end{equation}
In view of the energy estimate (\ref{esti}), it is easy to deduce from (\ref{3.5}) that
$$
\lim_{v'\rightarrow v} \underset{0\leq s\leq t}{sup} \left|X^v(s)-X^{v'}(s)\right|_H =0\quad \text{in\ Probability.}
$$
Since $\phi\in C_b(H)$,  it follows that
$$
\lim_{v'\rightarrow v}\mathbf{T_t}\phi(v') =\lim_{v'\rightarrow v}\mathbb{E}(\phi(X^{v'}(s)))=\mathbb{E}(\phi(X^{v}(s)))= \mathbf{T_t}\phi(v),
$$
proving the Feller property.                  \hfill $\blacksquare$
\end{proof}
\vspace{2em}
Denote by $\mathcal{P}(H)$ the set of all probability measures on $H$.\\
Then we have the following existence of invariant measure associated to $(\mathbf{T_t})_{t\geq0}$.
\begin{theorem}
Assume \textbf{(A.1)}-\textbf{(A.3)} hold, there is an invariant measure $\pi \in \mathcal{P}(H)$ for the associated semigroup $(\mathbf{T_t})_{t\geq0}$, i. e. for any $t\geq 0$ and $\phi\in C_b(H)$
$$\int_H \mathbf{T_t}\phi(u)\pi(du)=\int_H \phi(u)\pi(du).$$
\end{theorem}
\begin{proof}Let $x_0=0$. Using Ito's formula, we have
$$
\begin{aligned}
\mathbb{E}|X(t)|_H^2 +& 2\mathbb{E}\int_0^t (X(s), AX(s))ds\\
=&2\mathbb{E}\int_0^t \langle X(s), f(X(s))\rangle ds+2\mathbb{E}\int_0^t \langle X(s), L(ds)\rangle+\mathbb{E}\int_0^t |\sigma(X(s))|_{L_2(l^2,H)}^2ds.\\
\end{aligned}
$$
Let $\phi\equiv 0$ in (\ref{def5}) to get $\int_0^t \langle X(s), L(ds)\leq 0$. By  the Lipschitz conditions on the coefficients $f$, $\sigma$, we further deduce that
$${E}|X(t)|_H^2 +2\mathbb{E}\int_0^t \|X(s)\|^2ds\leq C_{C_1}\int_0^t\mathbb{E}(1+|X(s)|_H^2)ds\leq C_{C_1}\cdot t,$$
where we have used the fact that $X(s)\in \overline{D},s\geq 0.$
Therefore, for any $t\geq 0$,
\begin{equation}
\frac{1}{t}\int_0^t\mathbb{E} \|X(s)\|^2ds\leq C_{C_1}.
\end{equation}
Since the imbedding $V\subseteq H$ is compact,  the existence of an invariant measure $\pi$ now follows from the classical Krylov-Bogoliubov argument proving the tightness of the occupation measures $\{\Gamma_t(A)=\frac{1}{t}\int_0^t\chi_A (X(s))ds,t\geq 0\}$.(\cite{DZ2}).    \hfill $\blacksquare$
\end{proof}

\section{Exponential ergodicity}
\setcounter{equation}{0}
 \setcounter{definition}{0}
\subsection{A general setup}
Consider a Markov transition function $\{\mathbf{T_t} (x, A), x \in H, A \in \mathcal{B}(H)\}_{t\in \mathbb{R}_{+}}.$ We use the standard notation for the corresponding semigroup of integral operators
$$\mathbf{T_t} \varphi(x) =\int_H \varphi(y) \mathbf{T_t}  (x, dy), x \in H, t\in \mathbb{R}_{+}.$$
We also denote by $\{\mathbf{P}_x ,x\in H\}$ the corresponding Markov family, that is, $\mathbf{P}_x$ denotes the law of the Markov process $X =\{X(t) ,t \geq 0\}$ with $X_0 = x$ and the given transition function.

A function $d:H\times H \rightarrow \mathbb{R}_{+}$ is called \textit{distance-like} if it is symmetric, lower semicontinuous, and $d(x, y) = 0 \Leftrightarrow x = y.$  And for $\mu ,\nu\in \mathcal{P}(H)$, denote by $\mathcal{C}(\mu,\nu)$ the set of all couplings between $\mu$ and $\nu$, that is, the collection of all probability measures $\lambda(dx, dy)$ on the product space $H\times H$ with marginals $\mu$ and $\nu$.  For a given distance-like function $d$, the corresponding coupling distance $W_d:\mathcal{P}(H)\times \mathcal{P}(H)\rightarrow \mathcal{R}_{+}\cup\{\infty\}$ is defined by
$$ W_d(\mu,\nu):=\underset{\lambda\in \mathcal{C}(\mu,\nu)}{inf}\ \int _{H\times H}d(x,y)\lambda(dx,dy),\quad \mu,\nu \in \mathcal{P}(H).$$

If $d$ is a lower semicontinuous metric on $H$, then $W_d$ is the usual \textit{Warsserstein-1} distance. In particular, if $d$ is the discrete metric, that is, $d(x, y) = 1 (x\neq y)$, then $W_d$ coincides with the \textit{total variation distance} $d_{TV}$; the latter can also be defined as follows:
$$d_{TV}(\mu,\nu):=\underset{A\in\mathcal{B}(H)}{sup}\ |\mu(A)-\nu(A)|.$$

$\mathbf{T_t}(x,\cdot)$ is the law of $X(t)$ under  the probability $\mathbf{P}_x$.
\vskip 0.3cm

To begin with, recall the following proposition from \cite{BKS} and \cite{HMS}.
\begin{proposition}\label{criteria1}
 Suppose $(\mathbf{T_t})_{t\geq0}$ is a Feller semigroup. Assume there exists a measurable function $V:H\rightarrow [0,\infty)$ and a bounded distance-like function $d$ on $H$ such that the following conditions hold:\\
1.$V$ satisfies a Lyapunov condition, that is, there exist some $\gamma>0,K>0$ such that for any $t\geq 0, x\in H$
\begin{equation}
\mathbf{T_t}V(x)\leq V(x)-\gamma\int_0^t \mathbf{T_s}V(x)ds +Kt.
\end{equation}
2.There exist $0<t_1<t_2<\infty$ such that for all $t\in [t_1,t_2]$,
\begin{equation}\label{c2}
W_d(\mathbf{T_t}(x,\cdot),\mathbf{T_t}(y,\cdot))\leq d(x,y), \quad x,y\in H.
\end{equation}
3.There exists $t>0$ such that for any $M>0$ there exists $\varepsilon=\varepsilon(M,t)>0$ such that
\begin{equation}\label{c3}
W_d(\mathbf{T_t}(x,\cdot),\mathbf{T_t}(y,\cdot))\leq (1-\varepsilon) d(x,y), \quad x,y\in \{V\leq M\}.
\end{equation}\vspace{1em}
4.One has $\rho\wedge 1\leq d$, where $\rho(x,y):=|x-y|_H.$\\
Then the Markov semigroup $(\mathbf{T_t})_{t\geq0}$ has a unique invariant measure $\pi$. Moreover, there exist constants $C>0,r>0$ such that
$$W_d(\mathbf{T_t}(x,\cdot),\pi)\leq C(1+V(x))e^{-rt},\quad t\geq 0,x\in H.$$
\end{proposition}
\begin{remark}
If it is already known that the semigroup $(\mathbf{T_t})_{t\geq0}$  has an invariant measure, then condition 4 of Proposition \ref{criteria1} is not needed, see \cite{HMS}, Theorem 4.8.
\end{remark}

\begin{definition}
A distance-like function $d$ bounded by 1 is called \textit{contracting} for $\mathbf{T_t}$ if there exists $\alpha<1$ such that for any $x,y \in H$ with $d(x,y)<1$ we have
\begin{equation}\label{def4.1}
W_d(\mathbf{T_t}(x,\cdot),\mathbf{T_t}(y,\cdot))\leq \alpha d(x,y).
\end{equation}
\end{definition}

\begin{definition}
A set $B\subset H$ is called \textit{$d$-small} for $\mathbf{T_t}$ if for some $\varepsilon>0$
\begin{equation}\label{def4.2}
\underset{x,y\in B}{sup}\ W_d(\mathbf{T_t}(x,\cdot),\mathbf{T_t}(y,\cdot))\leq 1-\varepsilon.
\end{equation}
\end{definition}

Now for any fixed $\tilde{N}>0$ we consider the distance-like function
$$ d_{\tilde{N}}(x,y):=\tilde{N}|x-y|_H^{\frac{2\delta}{1+\delta}}\wedge 1,\quad \delta\in(0,1),\quad x,y\in H.$$
As a straightforward application of Proposition \ref{criteria1}, we have the following result.
\begin{proposition}\label{criteria}
  Suppose $(\mathbf{T_t})_{t\geq0}$ is a Feller semigroup. Let $d_{\tilde{N}}$  be defined as above. Assume there exists a measurable function $V:H\rightarrow [0,\infty)$ such that the following conditions hold:\\
(I).$V$ satisfies a Lyapunov condition, that is, there exist some $\gamma>0,K>0$ such that for any $t\geq 0, x\in H$
\begin{equation}\label{c1}
\mathbf{T_t}V(x)\leq V(x)-\gamma\int_0^t \mathbf{T_s}V(x)ds +Kt.
\end{equation}
(II).There exist $t_0>0$ and a locally bounded function $L(t)>0$, such that for any $t\geq t_0$ and $\tilde{N}\geq L(t)$, $d_{\tilde{N}}$ is contracting for $\mathbf{T_t}$.\\
(III).For any $\tilde{N}>0$, there exists $t_\ast>0$ such that for any $M>0,t\geq t_\ast$, the set $\{V\leq M\}$ is $d_{\tilde{N}}$-small for $\mathbf{T_{t}}$.\\
(IV).One has $\rho\wedge 1\leq  d_{\tilde{N}}$, where $\rho(x,y):=|x-y|_H.$\\
Then the Markov semigroup $(\mathbf{T_t})_{t\geq0}$ has a unique invariant measure $\pi$. Moreover, there exist constants $C>0,r>0,\tilde{N}>0$ such that
$$W_{d_{\tilde{N}}}(\mathbf{T_t}(x,\cdot),\pi)\leq C(1+V(x))e^{-rt},\quad t\geq 0,x\in H.$$
\end{proposition}
\begin{remark}
Clearly, $(\ref{c2})$ and $(\ref{c3})$ are satisfied if (II) and (III) hold.\\
Indeed, to prove $(\ref{c2})$, set $\tilde{N}:=(\sup_{t\in[t_0,2t_0]}L(t))\vee 1$.
Condition (II) implies that for some $\alpha<1$ we have
\begin{equation}\label{rmk1}
W_{d_{\tilde{N}}}(\mathbf{T_{t_0}}(x,\cdot),\mathbf{T_{t_0}}(y,\cdot))\leq  \alpha d_{\tilde{N}}(x,y),\quad \text{whenever } d_{\tilde{N}}(x,y)<1,
\end{equation}
and $d_{\tilde{N}}$ is contracting for $\mathbf{T_{t}}$ for any $t\in[t_0,2t_0]$, which follows that for any $t \in[t_0,2t_0]$, 
\begin{equation} \label{rmk3}
W_{d_{\tilde{N}}} (\mathbf{T_{t}}(x,\cdot),\mathbf{T_{t}}(y,\cdot))\leq  d_{\tilde{N}}(x,y),\quad x,y \in H.
\end{equation} 
Then $(\ref{rmk1})$ and $(\ref{rmk3})$ imply that 
for any $t\geq t_0$
\begin{equation} \label{rmk2}
W_{d_{\tilde{N}}} (\mathbf{T_{t}}(x,\cdot),\mathbf{T_{t}}(y,\cdot))\leq  d_{\tilde{N}}(x,y),\quad x,y \in H.
\end{equation}

To check $(\ref{c3})$, we pick $t:=t_0+ t_0\vee t_\ast$ and fix arbitrary $M>0$. It follows from condition (III) that the set $\{V\leq M\}$ is $d_{\tilde{N}}$-small for $\mathbf{T_{t_\ast}}$ with some $\varepsilon=\varepsilon(t,M)$. Now take any $x,y \in \{V\leq M\}$. If $d_{\tilde{N}}(x,y)<1$, using (\ref{rmk1}) and $(\ref{rmk2})$, we get
$$ W_{d_{\tilde{N}}}(\mathbf{T_{t_0+ t_0\vee t_\ast}}(x,\cdot),\mathbf{T_{t_0+t_0\vee  t_\ast}}(y,\cdot))\leq  W_{d_{\tilde{N}}}(\mathbf{T_{t_0}}(x,\cdot),\mathbf{T_{t_0}}(y,\cdot))\leq \alpha d_{\tilde{N}}(x,y).$$
If $d_{\tilde{N}}(x,y)=1$, $(\ref{rmk2})$ and $d_{\tilde{N}}$-small property implies
$$W_{d_{\tilde{N}}}(\mathbf{T_{t_0+t_0\vee  t_\ast}}(x,\cdot),\mathbf{T_{t_0+t_0\vee  t_\ast}}(y,\cdot))\leq W_{d_{\tilde{N}}}(\mathbf{T_{t_0\vee  t_\ast}}(x,\cdot),\mathbf{T_{t_0\vee  t_\ast}}(y,\cdot))\leq 1-\varepsilon=(1-\varepsilon)d_{\tilde{N}}(x,y).$$
Thus, $(\ref{c3})$ of Proposition \ref{criteria1} is met.
\end{remark}

\subsection{Main results}
In this section, we will establish the exponential ergodicity of the reflected stochastic evolution equation (\ref{SEE}) using a coupling method.\\

We assume that there exists an orthonormal basis $\{e_i, i\geq 1\}$ of $H$ consisting of the eigenvectors of $A$, i.e.
$$Ae_i=\lambda_i e_i,\quad \langle e_i,e_i\rangle =1,\quad i=1,2\cdots,$$
where $0<\lambda_1\leq \cdots \lambda_n \uparrow \infty$.

Recall the equation (\ref{SEE}):
\begin{equation}
\left\{
\begin{aligned} \label{4.1}
dX(t)&+AX(t)dt =f(X(t))dt +B(X(t),X(t))dt + \sum_{i=1}^\infty\sigma_i(X(t))d\beta_i(t)+dL(t), t\geq0,\\
 X(0)& =x_0, \quad x_0\in \overline{D}.
\end{aligned}
\right.
\end{equation}
Regarding the diffusion-coefficients, we introduce the following assumption.
\vskip 0.3cm
\textbf{H.1}: There exists a large integer $N$, such that
\begin{itemize}
\item[(i)]
\begin{equation}
\ \lambda_{N+1}>\frac{32}{3}|f(0)|_{V\ast}+ \frac{32}{3}|\sigma(0)|_{L_2(l^2,H)}^2 +16C_1.\hspace{4.5em}
\end{equation}
\item[(ii)]
 $\  H_N:=P_NH\subset Range(\sigma(x)):=span(\sigma_i(x),i=1,2,\cdots),$ and the corresponding pseudo-inverse operator $\sigma(x)^{-1}:H_N\rightarrow l^2$ is uniformly bounded over $x\in H$. Where $P_N$ stands for the projection to the linear span of the first $N$ eigenvectors $e_1,\cdots,e_N$ of the operator $A$.
 \end{itemize}
\begin{remark}
The condition $\lambda_{N+1}>\frac{32}{3}|f(0)|_{V\ast}+ \frac{32}{3}|\sigma(0)|_{L_2(l^2,H)}^2 +16C_1$ is not optimal.
\end{remark}
For given $x,y\in \overline{D}$, let $(X^x,L^x)$ be the solution of equation (\ref{4.1}) starting at $x$ and define $(Y^y,L^y)$ as the solution to the following SEE with reflection:
\begin{equation}\label{4.3}
\left\{
\begin{aligned}
dY^y(t)&+AY^y(t)dt =f(Y^y(t))dt +B(Y^y(t),Y^y(t))dt + \sum_{i=1}^\infty\sigma_i(Y^y(t))d\beta_i(t)\\
&\hspace{10em} +\frac{\lambda_{N+1}}{2}P_N(X^x(t)-Y^y(t))dt +dL^y(t), t\geq0,\\
 Y^y(0)& =y, \quad y\in \overline{D}.
\end{aligned}
\right.
\end{equation}

Indeed, one can interpret equation (\ref{4.3}) as an analogue to equation (\ref{4.1})
%
 with $dW (t)$ replaced by
$$
\begin{aligned}
&dW^{x,y}(t) := dW (t) + \beta^{x,y}(t) dt,\\
&\beta^{x,y}(t) := \frac{\lambda_{N+1}}{2}\sigma(Y^y)^{-1}P_N(X^x(t)-Y^y(t)).
\end{aligned}
$$
By \textbf{H.1}, the pseudo-inverse operator $\sigma(x)^{-1}:H_N\rightarrow l^2$ is uniformly bounded over $x\in H$; thus there exists a constant $C>0$ such that for all $t\geq 0$,
\begin{equation}
\|\beta^{x,y}(t)\|_{l^2}\leq C|P_N(X^x(t)-Y^y(t))|_H\leq C|X^x(t)-Y^y(t)|_H.
\end{equation}
Because the solution to equation (\ref{4.1}) is a probabilistically strong solution,   $X^y(t)$ is an image of the driving noise under a measurable mapping
$$\Phi^y_t:\mathcal{C}([0,t],\mathbb{R}^\infty)\rightarrow H.$$
In other words, $X^y(t)=\Phi^y_t(W_{[0,t]})$, where $W_{[0,t]}$ denotes the part of the trajectory $\{W(s),s\in[0,t]\}$. On the other hand, it follows from the Girsanov theorem (\cite{LS}, Theorem 7.4) that Law($W^{x,y}_{[0,t]}$) is absolutely continuous with respect to Law($W_{[0,t]}$). Therefore, by the uniqueness of the solution, we have $Y^y(t) =\Phi^y_t(W^{x,y}_{[0,t]})$.

\vskip 0.3cm
Here is the main result of this section.

\begin{theorem}\label{ergodicity}
Assume \textbf{(A.1)}-\textbf{(A.3)} and \textbf{(H.1)} hold. Then the reflected SEE (\ref{SEE}) has a unique invariant measure $\pi$. Furthermore, there exist a constants $\delta\in(0,1)$, $\tilde{N}>0$, $C>0,r>0$ such that
$$W_{d_{\tilde{N}}}(Law(X^x(t)),\pi)\leq C(1+|x|^2_H)e^{-rt},\quad t\geq 0,x\in H,$$
with $d_{\tilde{N}}(x,y)=\tilde{N}|x-y|_H^{\frac{2\delta}{1+\delta}}\wedge 1$.
\end{theorem}

To prove Theorem \ref{ergodicity}, we will verify the conditions in Proposition \ref{criteria}.  To begin with, we will establish some estimates about $X^x$ and $Y^y$.
\begin{lemma}\label{lem4.1}
If a continuous real-valued function $g$ satisfies that for some $\alpha>0,$ $g(t)-g(s)\leq -\alpha \int_s^t g(u)du, \forall t>s\geq 0$, then $g(t)\leq g(0)exp(-\alpha t),\forall t\geq 0.$
\end{lemma}
\begin{proof}
We extend $g$ to $\mathbb{R}$ by setting
$$
\overline{g}(t)=\left\{\begin{array}{rcl}
&g(t), & \text{ if }t\geq 0\\
&g(0)\text{exp}(-\alpha t), & \text{ if }t<0.
\end{array}\right.
$$
Then $\overline{g}(t)-\overline{g}(s)\leq -\alpha \int_s^t \overline{g}(u)du, \forall t>s.$\\
Let $\eta$ be a standard mollifier on $\mathbb{R}$, i.e.
$$
\eta(x)=\left\{\begin{array}{rcl}
&C\text{exp}\left(\frac{1}{|x|^2-1}\right), & \text{ if }|x|<1\\
&0, & \text{ if }|x|\geq 1,
\end{array}\right.
$$
the constant $C>0$ is selected so that $\int_\mathbb{R} \eta dx =1.$
For each $\varepsilon>0$, set $\eta_\varepsilon(x)=\frac{1}{\varepsilon}\eta(\frac{x}{\varepsilon}).$
Define the mollification $g^\varepsilon(t)=\int_\mathbb{R}\overline{g}(t-u)\eta_\varepsilon(u)du$ on $\mathbb{R}$. Then for $t>s\geq 0$:
$$
\begin{aligned}
g^\varepsilon(t)-g^\varepsilon(s)&=\int_{\mathbb{R}}(\overline{g}(t-u)-\overline{g}(s-u))\eta_\varepsilon(u)du\\
&\leq \int_{\mathbb{R}}(-\alpha \int_{s-u}^{t-u} \overline{g}(l)dl)\eta_\varepsilon(u)du\\
&= -\alpha \int_{\mathbb{R}}(\int_s^t \overline{g}(l-u)dl)\eta_\varepsilon(u)du\\
&=-\alpha\int_s^t g^\varepsilon(l)dl.
\end{aligned}
$$
It follows that for $t\geq 0$.
$$ (g^\varepsilon)' (t)\leq -\alpha g^\varepsilon(t).$$
Hence, we get  $g^\varepsilon(t)\leq g^\varepsilon(0)exp(-\alpha t),\  for\  t\geq 0.$ Then let $\varepsilon\rightarrow 0$ to get
$$g(t)\leq g(0)exp(-\alpha t).$$             \hfill $\blacksquare$
\end{proof}
\vspace{1em}
Next result is an estimate of the weighted moment of the difference $X^x(t)-Y^y(t)$.
\begin{proposition}\label{prop4.3}
Assume \textbf{(A.1)}-\textbf{(A.3)} and \textbf{(H.1)} hold. Then the reflected SEEs (\ref{4.1}) and (\ref{4.3}) satisfy
\begin{equation}
\mathbb{E}\left[ \text{exp}(-4	\int_0^t\|X^x(s)\|^2ds)|X^x(t)-Y^y(t)|_H^2\right]\leq \text{exp}\left\{\left(4C_1-\frac{3\lambda_{N+1}}{4}\right)t\right\}|x-y|^2_H, \quad t\geq 0,
\end{equation}
\end{proposition}
with $4C_1-\frac{3\lambda_{N+1}}{4}<0$.\\
\begin{proof}
Define $h(t)=\text{exp}(-4\int_0^t\|X^x(l)\|^2dl)$ and for $s<t$:
\begin{equation}\label{4.10}
\begin{aligned}
h(t)|X^x(t)-Y^y(t)|_H^2-&h(s)|X^x(s)-Y^y(s)|_H^2 +2\int_s^t h(l)(X^x(l)-Y^y(l), A(X^x(l)-Y^y(l)))dl\\
=& -4\int_s^t h(l)\|X^x(l)\|^2|X^x(l)-Y^y(l)|_H^2 dl\\
 & +2\sum_{i=1}^\infty\int_s^t h(l)\langle X^x(l)-Y^y(l), \sigma_i(X^x(l)) -\sigma_i(Y^y(l))\rangle d\beta_i(l)\\
 & +2\int_s^t h(l)\langle X^x(l)-Y^y(l), f(X^x(l))-f(Y^y(l))\rangle dl\\
 & +2\int_s^t h(l)\langle X^x(l)-Y^y(l), B(X^x(l))-B(Y^y(l))\rangle dl\\
 & -\lambda_{N+1}\int_s^t h(l)(X^x(t)-Y^y(t),P_N(X^x(t)-Y^y(t)))dl\\
 & +\int_s^t h(l)|\sigma(X^x(l))-\sigma(Y^y(l))|_{L_2(l^2,H)}^2dl\\
 & +2\int_s^t h(l)( X^x(l)-Y^y(l), L^x(dl))\\
 & -2\int_s^t h(l)( X^x(l)-Y^y(l), L^y(dl))\\
 &\hspace{-9em}:= I_1(s,t)+I_2(s,t)+I_3(s,t)+I_4(s,t)+I_5(s,t)+I_6(s,t)+I_7(s,t)+I_8(s,t).
\end{aligned}
\end{equation}
Observe that
$$2\int_s^t h(l)(X^x(l)-Y^y(l), A(X^x(l)-Y^y(l)))dl=2\int_s^t h(l)\|X^x(l)-Y^y(l)\|^2
dl.$$
By the assumption on $f$ and Young's inequality, we have
$$
\begin{aligned}
I_3(s,t)&\leq 2\int_s^t h(l)\|X^x(l)-Y^y(l)\| |f(X^x(l))-f(Y^y(l))|_{V^\ast} dl\\
 &\leq\frac{1}{4}\int_s^t h(l)\|X^x(l)-Y^y(l)\|^2dl + 4\int_s^t h(l)  |f(X^x(l))-f(Y^y(l))|_{V^\ast}^2 dl.
\end{aligned}
$$
By assumption \textbf{(A.1)}, we have
$$
\begin{aligned}
I_4(s,t)&\leq 2\int_s^t h(l)|\overline{b}(X^x(l),X^x(l),X^x(l)-Y^y(l))-\overline{b}(Y^y(l),Y^y(l),X^x(l)-Y^y(l))|dl\\
 &= 2\int_s^t h(l)|\overline{b}(X^x(l)-Y^y(l),X^x(l),X^x(l)-Y^y(l))|dl\\
 &\leq 4\int_s^t h(l)|X^x(l)-Y^y(l)|_H\|X^x(l)\|\|X^x(l)-Y^y(l)\|dl\\
 &\leq \int_s^t h(l)\|X^x(l)-Y^y(l)\|^2dl + 4\int_s^t h(l)|X^x(l)-Y^y(l)|^2_H\|X^x(l)\|^2dl.
\end{aligned}
$$
On the other hand,
$$I_3(s,t)+I_6(s,t)\leq\frac{1}{4}\int_s^t h(l)\|X^x(l)-Y^y(l)\|^2dl + 4C_1\int_s^t h(l)  |X^x(l)-Y^y(l)|_H^2 dldl.$$
As for $I_7(s,t)$, recalling the fact that $Y^y$ is also a $\overline{D}$-valued continuous process, we can follow the similar arguments as in \cite{BZ} ((4.29)-(4.35)) to prove
$$\int_s^t h(l)( X^x(l)-Y^y(l), L^x(dl))=\underset{n\rightarrow\infty}{lim}\ -n \int_s^t h(l)( X^{x,n}(l)-Y^y(l), X^{x,n}(l)-\Pi(X^{x,n}(l)))dl\leq 0,$$
where $X^{x,n}$ satisfies the following penalized stochastic evolution equations:
\begin{equation}
\left\{
\begin{aligned}
dX^{x,n}(t)&+AX^{x,n}(t)dt =f(X^{x,n}(t))dt +B(X^{x,n}(t),X^{x,n}(t))dt \\
&\hspace{6.5em}+ \sum_{i=1}^\infty\sigma_i(X^{x,n}(t))d\beta_i(t)-n(X^{x,n}(t)-\Pi(X^{x,n}(t)))dt, t\geq0,\\
 X^{x,n}(0)& =x, \quad x\in \overline{D}.
\end{aligned}
\right.
\end{equation}
with
$$
\Pi(y)=\left\{\begin{array}{rcl}
&y, & \text{ if }|y|_H\leq 1, y\in H\\
&\frac{y}{|y|_H}, & \text{ if }|y|_H\geq 1,y\in H.
\end{array}\right.
$$
Similarly, we have $I_8(s,t)\leq 0$.
Substituting above estimates into (\ref{4.10}) and using the fact that $I_2(s,t)$ is a martingale, we obtain
$$
\begin{aligned}
&\mathbb{E}[h(t)|X^x(t)-Y^y(t)|_H^2]-\mathbb{E}[h(s)|X^x(s)-Y^y(s)|_H^2] +\frac{3}{4}\mathbb{E}[\int_s^t h(l)\|X^x(l)-Y^y(l)\|^2dl]\\
&\hspace{3.5em}\leq\mathbb{E}[4C_1\int_s^t h(l) |X^x(l)-Y^y(l)|^2_H dl]-\mathbb{E}[\lambda_{N+1}\int_s^t h(l)|P_N(X^x(t)-Y^y(t))|^2_Hdl].
\end{aligned}
$$
Combining with the Poincar\'{e} inequality:
$$
\|X^x(l)-Y^y(l)\|^2\geq \lambda_{N+1}|(I-P_N)(X^x(l)-Y^y(l))|^2_H,
$$
we obtain that
\begin{equation}\label{4.12}
\begin{aligned}
&\mathbb{E}[h(t)|X^x(t)-Y^y(t)|_H^2]-\mathbb{E}[h(s)|X^x(s)-Y^y(s)|_H^2]\\
&\hspace{8em}\leq\mathbb{E}[(4C_1-\frac{3\lambda_{N+1}}{4})\int_s^t h(l) |X^x(l)-Y^y(l)|^2_H dl],
\end{aligned}
\end{equation}
with $4C_1-\frac{3\lambda_{N+1}}{4}<0$.
Applying Lemma \ref{lem4.1} to (\ref{4.12}), we get the desired inequality,
$$\mathbb{E}[h(t)|X^x(t)-Y^y(t)|_H^2]\leq |x-y|^2_H  \text{exp}\left\{\left(5C_1-\frac{4\lambda_{N+1}}{5}\right)t\right\}, \quad t\geq 0.$$  \hfill $\blacksquare$
\end{proof}
We also need the following exponential integrability estimate.
\begin{proposition}\label{exint}
Assume \textbf{(A.1)}-\textbf{(A.3)} hold. Then, for any $t>0,\delta>0$
\begin{equation}
\mathbb{E}\left[\text{exp}(\hspace{0.1em}4\delta\int_0^t\|X^x(s)\|^2ds)\right]\leq \text{exp}\left\{4\delta+\left(8\delta|f(0)|_{V^\ast}^2+(8\delta+64\delta^2)|\sigma(0)|^2_{L_2(l^2,H)}+(8\delta+64\delta^2)C_1\right)t\right\}.
\end{equation}
\end{proposition}
\begin{proof}
Applying the Ito's formula to $|X^x|^2_H$, we have
\begin{align}\label{4.12-1}
\begin{aligned}
|X^x(t)|^2_H+2\int_0^t(AX^x(s),X^x(s))ds=&\ |x|^2_H+2\int_0^t\langle f(X^x(s)),X^x(s)\rangle ds\\
 &+\sum_{i=1}^\infty2\int_0^t(\sigma_i(X^x(s)),X^x(s)) d\beta_i(s)\\
 &+\int_0^t|\sigma(X^x(s))|^2_{L_2(l^2,H)}ds+2\int_0^t(dL^x(s),X^x(s)).\\
\end{aligned}
\end{align}
Clearly, $2\int_0^t(AX^x(s),X^x(s))ds=2\int_0^t\|X^x(s)\|^2ds$ and $2\int_0^t(dL^x(s),X^x(s))\leq 0$.\\
Using the assumption on $f$ and $\sigma$:
\begin{align}\label{4.12-2}
\begin{aligned}
2\int_0^t&\langle f(X^x(s)),X^x(s)\rangle ds+\int_0^t|\sigma(X^x(s))|^2_{L_2(l^2,H)}ds\\
&\leq 2\int_0^t|f(X^x(s))|_{V^\ast}\|X^x(s)\|ds +\int_0^t(|\sigma(0)|_{L_2(l^2,H)}+|\sigma(X^x(s)-\sigma(0)|_{L_2(l^2,H)})^2ds\\
 &\leq \int_0^t\|X^x(s)\|^2ds +\int_0^t|f(X^x(s))|_{V^\ast}^2ds+\int_0^t(|\sigma(0)|_{L_2(l^2,H)}+|\sigma(X^x(s)-\sigma(0)|_{L_2(l^2,H)})^2ds\\
 &\leq \int_0^t\|X^x(s)\|^2ds+ 2\int_0^t (|\sigma(0)|^2_{L_2(l^2,H)}+|f(0)|_{V^\ast}^2+C_1|X^x(s)|_H^2)ds.\\
\end{aligned}
\end{align}
Substitute (\ref{4.12-2}) into (\ref{4.12-1}) and arrange the terms to get 
\begin{align}\label{4.12-3}
\begin{aligned}
4\delta\int_0^t\|X^x(s)\|^2ds\leq&\ 4\delta|x|^2_H+8\delta\int_0^t(|f(0)|^2_{V^\ast}+|\sigma(0)|^2_{L_2(l^2,H)} +C_1|X^x(s)|^2_H )ds\\
 &+8\delta\sum_{i=1}^\infty \int_0^t(\sigma_i(X^x(s)),X^x(s)) d\beta_i(s)\\
\end{aligned}
\end{align}
Since $X^x$ is a $\overline{D}$-valued continuous process, we have 
$$
\mathbb{E}\left\{\text{exp}\left(32\delta^2\int_0^t \|(\sigma(X^x(s)),X^x(s))\|_{l^2}\right)\right\}< \infty.
$$
By Novikov's criterion,
$$
t\rightarrow \text{exp}\left\{8\delta\sum_{i=1}^\infty \int_0^t(\sigma_i(X^x(s)),X^x(s)) d\beta_i(s)-32\delta^2\int_0^t \|(\sigma(X^x(s)),X^x(s))\|^2_{l^2}\right\}=:\mathcal{E}(t)
$$
is an exponential martingale.\\
Therefore, recalling $x\in \overline{D}$ and taking expectation it follows from (\ref{4.12-3}) that
$$
\begin{aligned}
&\mathbb{E}\left[\text{exp}(\hspace{0.1em}4\delta\int_0^t\|X^x(s)\|^2ds)\right]\\
&\hspace{2em}\leq \mathbb{E}\left\{\text{exp}\left(4\delta+\left(8\delta(|f(0)|^2_{V^\ast}+|\sigma(0)|^2_{L_2(l^2,H)} +C_1 )+64\delta^2(|\sigma(0)|^2_{L_2(l^2,H)} +C_1)\right)t\right)\cdot\mathcal{E}(t)\right\}\\
&\hspace{2em}\leq\text{exp}\left\{4\delta+\left(8\delta|f(0)|_{V^\ast}^2+(8\delta+64\delta^2)|\sigma(0)|^2_{L_2(l^2,H)}+(8\delta+64\delta^2)C_1\right)t\right\}.
\end{aligned}
$$            
\hfill $\blacksquare$
\end{proof}
\vspace{1em}
Recall that $X^x, Y^y$ are the solutions respectively to the reflected SEEs (\ref{4.1}) and (\ref{4.3}).
\begin{proposition}\label{prop4.5}
Assume \textbf{(A.1)}-\textbf{(A.3)} and \textbf{(H.1)} hold. Then 
$$\mathbb{E}[\exp(-8\int_0^t\|X^x(l)\|^2dl)|X^x(t)-Y^y(t)|_H^4]\leq C_t|x-y|^4_H,\quad t\geq 0$$
where $C_t$ is locally bounded w.r.t. $t$.
\end{proposition}

\begin{proof}
Define $g(t)=\exp(-8\int_0^t\|X^x(l)\|^2dl)$ and apply Ito's formula to $g(t)|X^x(t)-Y^y(t)|_H^4.$ The Proof is completely analogous to the proof of Proposition \ref{prop4.3} 

\hfill $\blacksquare$

\end{proof}

\vskip 0.3cm
\noindent\textbf{Proof of theorem \ref{ergodicity}}

We will verify that all the conditions (I)-(IV) in Proposition \ref{criteria} are satisfied\\
For $t>0$, we define the semigroup $\mathbf{T_t}$ with $\{X^x(t):x\in \overline{D}, t\geq 0]\}$ by
$$ \mathbf{T_t}\phi(x):=\mathbb{E}(\phi(X^x(t))),\quad \phi \in C_b(H).$$
As proved in Proposition \ref{exint},
$$
\begin{aligned}
|X^x(t)|^2_H +\int_0^t\|X^x(s)\|^2ds\leq&\ |x|^2_H+2\int_0^t(|f(0)|^2_{V^\ast}+|\sigma(0)|^2_{L_2(l^2,H)} +2M^2|X^x(s)|^2_H )ds\\
 &+\sum_{i=1}^\infty 2\int_0^t(\sigma_i(X^x(s)),X^x(s)) d\beta_i(s)\\
\end{aligned}
$$
Take expectation and use the Poincare inequality to obtain
$$
\begin{aligned}
\mathbb{E}[|X^x(t)|^2_H]\leq&\ |x|^2_H-\lambda_1\mathbb{E}[\int_0^t|X^x(s)|^2_Hds]+2(|f(0)|^2_{V^\ast}+|\sigma(0)|^2_{L_2(l^2,H)}+2C_1)t\\.
\end{aligned}
$$
So condition (I) holds with $V(x)=|x|^2_H.$\vspace{1em}\\
Define $d_{\tilde{N}}(x,y)=\tilde{N}|x-y|_H^{\frac{2\delta}{1+\delta}}\wedge 1$. Take any $x,y\in H$ with $d_{\tilde{N}}(x,y)<1$. In this case,  $d_{\tilde{N}}(x,y)=\tilde{N}|x-y|_H^{\frac{2\delta}{1+\delta}}$. By coupling lemma (\cite{V} theorem 4.1), there exists a pair of random variables $\hat{Y},Z$ such that Law($\hat{Y}$)=Law($Y^y(t)$), Law($Z$)=$\mathbf{T_t}(y,\cdot)$ and
$$
\begin{aligned}
P(\hat{Y}\neq Z)&=d_{TV}(Law(Y^y(t)),Law(X^y(t)))\\
&=d_{TV}(Law(\Phi^y_t(W^{x,y}_{[0,t]})),Law(\Phi^y_t(W_{[0,t]})))\\
&\leq d_{TV}(Law(W^{x,y}_{[0,t]}),Law(W_{[0,t]}))\\
&\leq C_\delta\left(\mathbb{E}\left(\int_0^t |X^x(s)-Y^y(s)|_H^2ds\right)^\delta\right)^{\frac{1}{1+\delta}},\quad \forall \delta\in(0,1),
\end{aligned}
$$
where the last inequality follows from \cite{BKS} Theorem A.5.
Then by Jessen inequality, Proposition \ref{exint} and Proposition \ref{prop4.5} we have

\begin{equation}\label{YneqZ}
\begin{aligned}
&P(\hat{Y}\neq Z)\\
&\leq C_\delta\left(\mathbb{E}\int_0^t |X^x(s)-Y^y(s)|_H^2ds \right)^{\frac{\delta}{1+\delta}}\\
&= C_\delta\left(\int_0^t \mathbb{E}\left(\exp(4\int_0^s\|X^x(l)\|^2dl)\exp(-4\int_0^s\|X^x(l)\|^2dl)|X^x(s)-Y^y(s)|_H^2\right)ds\right)^{\frac{\delta}{1+\delta}}\\
&\leq C_\delta\left(\int_0^t \mathbb{E}\left(\exp(8\int_0^s\|X^x(l)\|^2dl)\right)^{\frac{1}{2}}\mathbb{E}\left(\exp(-8\int_0^s\|X^x(l)\|^2dl)|X^x(s)-Y^y(s)|_H^4\right)^{\frac{1}{2}}ds\right)^{\frac{\delta}{1+\delta}}                 \\
&\leq C_{\delta,t}|x-y|^{\frac{2\delta}{1+\delta}}_H,\quad \forall \delta\in(0,1),
\end{aligned}
\end{equation}
where coefficient $C_{\delta,t}$ is locally bounded w.r.t. $t$.

By the gluing lemma (see, e.g., \cite{K}, Lemma 4.3.2, and \cite{V}, p. 23), we can assume that $\hat{Y},Z$ are defined on the same probability space as $X^x(t)$ and $Y^y(t)$, and $\hat{Y}$=$Y^y(t)$. Then the joint
law of $X^x(t),Z$ is a coupling for $\mathbf{T_t} (x, \cdot), \mathbf{T_t}(y, \cdot)$, and thus

\begin{align}\label{4.12-4}
\begin{aligned}
&W_{d_{\tilde{N}}}(\mathbf{T_t} (x, \cdot), \mathbf{T_t}(y, \cdot))\\
&\leq \mathbb{E}d_{\tilde{N}}(X^x(t),Z)\\
&=\mathbb{E}d_{\tilde{N}}(X^x(t),Z)\mathbb{I}(Y^y(t)=Z)+\mathbb{E}d_{\tilde{N}}(X^x(t),Z)\mathbb{I}(Y^y(t)\neq Z)\\
&\leq \mathbb{E}d_{\tilde{N}}(X^x(t),Y^y(t))+ P(Y^y(t)\neq Z)\\
&:= T_1 +T_2.
\end{aligned}
\end{align}
By H\"{o}lder inequality,
$$
\begin{aligned}
T_1&=\mathbb{E}d_{\tilde{N}}(X^x(t),Y^y(t))= \mathbb{E}\left[\tilde{N}|X^x(t)-Y^y(t)|_H^{\frac{2\delta}{1+\delta}}\wedge 1\right]\\
&\leq \tilde{N}\mathbb{E}\left[\left(\text{exp}(2\int_0^t\|X^x(s)\|^2ds)\text{exp}(-2\int_0^t\|X^x(s)\|^2ds)|X^x(t)-Y^y(t)|_H\right)^{\frac{2\delta}{1+\delta}}\right]\\
&=\tilde{N}\mathbb{E}\left[\left(\text{exp}(4\int_0^t\|X^x(s)\|^2ds)\right)^{\frac{\delta}{1+\delta}}\left(\text{exp}(-4\int_0^t\|X^x(s)\|^2ds)|X^x(t)-Y^y(t)|^2_H\right)^{\frac{\delta}{1+\delta}}\right]\\
&\leq \tilde{N}\left(\mathbb{E}\left(\text{exp}(4\delta\int_0^t\|X^x(s)\|^2ds)\right)\right)^{\frac{1}{1+\delta}}            \left(\mathbb{E}\left(\text{exp}(-4\int_0^t\|X^x(s)\|^2ds)|X^x(t)-Y^y(t)|^2_H\right)\right)^{\frac{\delta}{1+\delta}}\\
\end{aligned}
$$
Using Proposition \ref{exint} and Proposition \ref{prop4.3},
$$
\begin{aligned}
&\left(\mathbb{E}\left(\text{exp}(4\delta\int_0^t\|X^x(s)\|^2ds)\right)\right)^{\frac{1}{1+\delta}}\\
\leq&\ \text{exp}\left(\frac{4\delta}{1+\delta}+\left(\frac{8\delta}{1+\delta}|f(0)|_{V^\ast}^2+\frac{8\delta+64\delta^2}{1+\delta}|\sigma(0)|^2_{L_2(l^2,H)}+\frac{8\delta+64\delta^2}{1+\delta}C_1\right)t\right).
\end{aligned}
$$
$$
\begin{aligned}
&\hspace{-8.5em}\left(\mathbb{E}\left(\text{exp}(-4\int_0^t\|X^x(s)\|^2ds)|X^x(t)-Y^y(t)|^2_H\right)\right)^{\frac{\delta}{1+\delta}}\\
&\hspace{-10em}\leq|x-y|^{\frac{2\delta}{1+\delta}}_H\text{exp}\left((\frac{4\delta}{(1+\delta)}C_1-\frac{3\delta}{4(1+\delta)}\lambda_{N+1})t\right).
\end{aligned}
$$
Putting the above estimates together, we arrive at 
$$
\begin{aligned}
T_1\leq &\tilde{N}|x-y|^{\frac{2\delta}{1+\delta}}_H\\
&\times\text{exp} \left(4\delta
 +\left(8\delta|f(0)|_{V^\ast}^2+(8\delta+64\delta^2)|\sigma(0)|^2_{L_2(l^2,H)}+(64\delta^2+12\delta)C_1-\frac{3}{4}\delta\lambda_{N+1}\right)\frac{t}{1+\delta}\right).\\
\end{aligned}
$$
Due to the assumption  $\lambda_{N+1}>\frac{32}{3}|f(0)|_{V\ast}+ \frac{32}{3}|\sigma(0)|_{L_2(l^2,H)}^2 +16C_1$, there exists $\delta\in(0,1)$ such that $$8\delta|f(0)|_{V^\ast}^2+(8\delta+64\delta^2)|\sigma(0)|^2_{L_2(l^2,H)}+(64\delta^2+12\delta)C_1-\frac{3}{4}\delta\lambda_{N+1}<0.$$
Fix any $\delta$ that satisfies the above inequality. Then we can choose $t_0$ large enough such that $\text{for any }t\geq t_0,\ T_1 \leq \frac{1}{2}\tilde{N}|x-y|_H^{\frac{2\delta}{1+\delta}}=\frac{1}{2}d_{\tilde{N}}(x,y)$.
\vspace{0.2em}\\
If we choose $L(t)>6C_{\delta,t}$ to be a locally bounded function(e.g., $L(t)=7C_{\delta,t}$), then by $(\ref{YneqZ})$, for any $\tilde{N}\geq L(t)$.
$$
\begin{aligned}
T_2&= P(Y^y(t)\neq Z)= P(\hat{Y}\neq Z)\\
&\leq C_{\delta,t}|x-y|^{\frac{2\delta}{1+\delta}}_H\\
&\leq \frac{1}{6}\tilde{N} |x-y|^{\frac{2\delta}{1+\delta}}_H=\frac{1}{6}d_{\tilde{N}}(x,y).
\end{aligned}
$$
Then we deduce from (\ref{4.12-4}) that
$$W_{d_{\tilde{N}}}(\mathbf{T_t} (x, \cdot), \mathbf{T_t}(y, \cdot))\leq \frac{2}{3} d(x,y),$$
which means that $d_{\tilde{N}}$ is contracting for $\mathbf{T_t}$, verifying the condition (II).\vspace{2em}\\
By \cite{BKS} theorem A.5 (A.14), there exists $\varepsilon >0$, such that
$$
\begin{aligned}
P(\hat{Y}\neq Z)&=d_{TV}(Law(Y^y(t)),Law(X^y(t)))\\
&=d_{TV}(Law(\Phi^y_t(W^{x,y}_{[0,t]})),Law(\Phi^y_t(W_{[0,t]})))\\
&\leq d_{TV}(Law(W^{x,y}_{[0,t]}),Law(W_{[0,t]}))\\
&\leq 1-\varepsilon.
\end{aligned}
$$
For $M>0$ and any $x,y\in \{\mathcal{V}\leq M\}$, Similar as before, we have
$$
\begin{aligned}
&W_{d_{\tilde{N}}}(\mathbf{T_t} (x, \cdot), \mathbf{T_t}(y, \cdot))\\
&\leq \mathbb{E}d_{\tilde{N}}(X^x(t),Z)\\
&=\mathbb{E}d_{\tilde{N}}(X^x(t),Z)\mathbb{I}(Y^y(t)=Z)+\mathbb{E}d_{\tilde{N}}(X^x(t),Z)\mathbb{I}(Y^y(t)\neq Z)\\
&\leq \mathbb{E}d_{\tilde{N}}(X^x(t),Y^y(t))+ P(\hat{Y}\neq Z)\\
&\leq \mathbb{E}d_{\tilde{N}}(X^x(t),Y^y(t))+ 1-\varepsilon,\\
\end{aligned}
$$
where
$$
\begin{aligned}
&\mathbb{E}d_{\tilde{N}}(X^x(t),Y^y(t))\\
&\leq\tilde{N}|x-y|^{\frac{2\delta}{1+\delta}}_H\\
&\hspace{0.5em}\times\text{exp} \left(4\delta
 +\left(8\delta|f(0)|_{V^\ast}^2+(8\delta+64\delta^2)|\sigma(0)|^2_{L_2(l^2,H)}+(64\delta^2+\frac{39}{2}\delta)C_1^2-3\delta\lambda_{N+1}\right)\frac{t}{1+\delta}\right).\\
\end{aligned}
$$
Recall the choice of $\lambda_{N+1}$ and  that $x,y\in \overline{D}$, so there exists $t_\ast>0$ such that for any $t\geq t_\ast$,
$$\underset{x,y\in \{V\leq M\}}{sup}\ \mathbb{E}d_{\tilde{N}}(X^x(t),Y^y(t))<\frac{\varepsilon}{2}.$$
Hence, for any $t\geq t_\ast$, $M>0$,
$$\underset{x,y\in \{V\leq M\}}{sup}\ W_{d_{\tilde{N}}}(\mathbf{T_t} (x, \cdot), \mathbf{T_t}(y, \cdot))<1-\frac{\varepsilon}{2},$$
completing the verification of the condition (III). Condition (IV) holds as long as $\tilde{N}\geq 1$.

\hfill $\blacksquare$

\section{Exponential ergodicity of Reflected Stochastic Navier Stokes Equations}
\setcounter{equation}{0}
 \setcounter{definition}{0}
As an application, in this section we obtain the exponential ergodicity of the stochastic Navier Stokes Equations on a two dimensional bounded domain $U$.

For a natural number $d$ and $p \in [1,\infty)$, let $U$ be a bounded open subset of $\mathbb{R}^d$ with $C^1$ boundary $\partial U$ and $L^p=L^p(U, \mathbb{R}^d)$ be the classical Lebesgue space of all $\mathbb{R}^d$-valued lebesgue measurable functions $v=(v^1,\cdots,v^d)$ defined on $U$ endowed with the following classical norm
\begin{equation}
\|v\|_{L^p}=\left(\sum_{k=1}^d\|v^k\|^p_{L^p(U)}\right)^{\frac{1}{p}}.
\end{equation}
For $p=\infty$, we set $\|v\|_{L^\infty}=\text{max}_{k=1}^d\|v^k\|_{L^{\infty}(U)}$.\\
Set $J^s=(I-\Delta)^{\frac{s}{2}}$. And define the generalized Sobolev spaces of divergence free vector distributions, for $s\in \mathbb{R}$, as
\begin{equation}
\begin{aligned}
H^{s,p}&=\{u\in \mathcal{S}'(U,\mathbb{R}^d):\|J^s u\|_{L^p}<\infty\},\\
H^{s,p}_{\text{sol}}&=\{ u\in H^{s,p}: \text{div }u = 0\}.
\end{aligned}
\end{equation}
It's well-known that $J^\sigma$ is an isomorphism between $H^{s,p}$ and $H^{s-\sigma,p}$ for $s\in \mathbb{R}$ and $ 1<p<\infty$. Moreover, $H^{s_2,p}\subset H^{s_1,p}$ when $s_1<s_2$. For the Hilbert case $p=2$, we set $H=H^{0,2}_{\text{sol}}$ and, for $s \neq 0,H^s=H^{0,2}_{\text{sol}}$, so that $H^s$ is a proper closed subspace of the classical Sobolev space usually denoted by the same symbol. In particular, we put
\begin{equation}
    H=\{v\in L^2(U,\mathbb{R}^d)):\text{div }v = 0\}.
\end{equation}
with scalar product inherited from $L^2(U, \mathbb{R}^d)$.

Denote by $\langle\cdot,\cdot\rangle$ the duality bracket between $(H^{s,p})'$ and $H^{s,p}$ spaces. Note that for $p \in [1,\infty)$, the space $(H^{s,p})'$ can be identified with $(H^{-s,p^{\ast}})$, where $\frac{1}{p}+\frac{1}{p^{\ast}}=1$.

Now we  define the operators appearing in the abstract formulation. Assume $s \in \mathbb{R}$ and $1\leq p <\infty$. Let $A_0=-\Delta$; then $A_0$ is a linear unbounded operator in $H^{s,p}$ and bounded from $H^{s+2,p}$ to $H^{s,p}$. Moreover, the space $H^{s,p}_{\text{sol}}$ are invariant w.r.t. $A_0$ and the corresponding operator will be denoted by $A$. We can observe that $A$ is a linear unbounded operator in $H^{s,p}$ and bounded from $H^{s+2,p}_{\text{sol}}$ to $H^{s,p}_{\text{sol}}$.  The operator $-A_0$ generates a contractive and analytic $C_0$-semigroup $\{e^{-tA}\}_{t\geq 0}$ on $H^{s,p}$ and therefore, the operator $-A$ generates a contractive and analytic $C_0$-semigroup $\{e^{-tA}\}_{t\geq 0}$  on $H^{s,p}_{\text{sol}}$. Moreover, for $t > 0$ the operator$e^{-tA}$ is bounded from $H^{s,p}_{\text{sol}}$ into $H^{s',p}_{\text{sol}}$
with $s' > s$ and there exists a constant $M$ (depending on $s' - s$ and $p$) such that
\begin{equation}
    \|e^{-tA}\|_{\mathcal{L}({H^{s,p}_{\text{sol}},H^{s',p}_{\text{sol}}})}\leq M(1+t^{-(s'-s)2}).
\end{equation}
We have $A:H^1\rightarrow H^{-1}$ as a linear bounded operator and
\begin{equation}
    \langle Av,v\rangle =\|\nabla v\|^2_{L^2}, \quad v \in H^1,
\end{equation}
Where
\begin{equation}
    \|\nabla v\|^2_{L^2}=\sum_{k-1}^d \|\nabla v^k\|^2_{L^2}, \quad v \in H^1.\\
\end{equation}

Moreover,
\begin{equation}
    \|v\|^2_{H^1}=\|v\|^2_{L^2} + \|\nabla v\|^2_{L^2}.
\end{equation}

We define a bounded trilinear form $b:H^1\times H^1\times H^1 \rightarrow \mathbb{R}$ by
\begin{equation}
    b(u,v,z)=\int_{\mathbb{R}^d} (u(\xi)\cdot \nabla)v(\xi)\cdot z(\xi) d\xi, \quad u,v,z \in H^1.
\end{equation}
and the corresponding bounded bilinear operator $B:H^1\times H^1\times H^1 \rightarrow H^{-1}$ via the trilinear form
\begin{equation}
    \langle B(u,v),z\rangle = b(u,v,z), \quad u,v,z \in H^1.
\end{equation}
This operator satisfies, for all $u,v,z \in H^1$,
\begin{equation}
    \langle B(u,v),z\rangle=-\langle B(u,z),v\rangle, \quad \langle B(u,v),v\rangle =0.
\end{equation}
Finally, we define the noise  forcing term. Given a real separable Hilbert space $K$ we consider a $K$-cylindrical Wiener process $\{W(t)\}_{t\geq0}$ defined on a stochastic basis $(\Omega,\mathcal{F,\mathbb{F}}=(\mathcal{F}_t)_{t\geq 0},\mathbb{P})$ satisfying the usual conditions. For the covariance $\sigma$ of the noise we make the following assumptions.

\textbf{(G1)} the mapping $\sigma : H\rightarrow\gamma(K;H)$  is well-defined and is a Lipschitz continuous map $G : H \rightarrow \gamma(K;H)$, i.e.
\begin{equation}
    \exists\  C_1>0 :\|\sigma(v_1)-\sigma(v_2)\|^2_{\gamma(K;H)}\leq C_1|v_1-v_2|_H^2
\end{equation}
for all $v_1,v_2\in H$, where $\gamma(K,H)$ denote by the space of all Hilbert-Schmidt operators from $K$ to $H$ and by $\|\cdot\|_{ \gamma(K,H)}$ we denote the corresponding Hilbert-Schmidt norm.
\vspace{2em}

\textbf{H.1}: There exists an integer $N$, such that of all $x\in H$, we have
$$ H_N:=P_NH\subset Range(\sigma(x))=span(\sigma_i(x),i=1,2,\cdots)$$
and

\begin{equation}
\lambda_{N+1}>\frac{32}{3}\|\sigma(0)\|_{\gamma(K;H)}^2+ 12C_1+ 16\gamma^2.
\end{equation}

We consider the stochastic damped Navier-Stokes equations, that is the equations of motion of a viscous incompressible fluid with two forcing terms, one is random and the other one is deterministic. These equations are
\begin{equation}\label {NV}
\left\{
\begin{aligned}
&\partial_t X + [-\nu \Delta X +\gamma X + (X\cdot \nabla)X + \nabla p] dt = \sigma(X) \partial_t W +f dt,\\
& \text{div} X =0,
\end{aligned}
\right.
\end{equation}
where the unknowns are the vector velocity $X = X(t,\xi)$ and the scalar pressure $p = p(t,\xi)$ for $t \geq 0$ and $\xi \in  \mathbb{R}^d$. By $\nu > 0$ we denote the kinematic viscosity and by $\gamma  > 0$ the sticky
viscosity. When $\gamma = 0$, (\ref{NV}) reduce to the classical stochastic
Navier-Stokes equations. The notation $\partial_t W$ on the right hand side is for the space correlated and white in time noise and $f$ is a deterministic forcing term. We consider a multiplicative term $\sigma(X)$ keeping track of the fact that the noise may depend on the velocity.
Projecting equations (\ref{NV}) onto the space $H$ of divergence free vector fields, we get the abstract form of the stochastic damped Navier-Stokes equations (\ref{NV}):
\begin{equation}\label{NV2}
    dX^x(t)+[AX^x(t)dt+\gamma X^x(t) + B(X^x(t),X^x(t))]dt = \sigma(X^x(t)) dW(t) + f(t)dt
\end{equation}
with the initial condition $ X^x(0)=x,$ where the initial velocity $x:\Omega \rightarrow H$ is an $\mathcal{F}_0$-measurable random variable .  Here $\gamma >0 $ is fixed and for simplicity we have put $\nu =1$.

Now consider the reflected  stochastic Navier-Stokes equation:
\begin{equation}\label{RNV}
   dX^x(t)+[AX^x(t)dt+\gamma X^x(t) + B(X^x(t),X^x(t))]dt = \sigma(X^x(t)) dW(t) + f(t)dt +dL^x(t)
\end{equation}
with the initial condition $ X^x(0)=x\in \overline{D}$. And we define $(Y^y,L^y)$ as the solution to the following SEE with the initial value $y\in\overline{D}$:
$$
\begin{aligned}
dY^y(t)+[AY^y(t)dt+\gamma Y^y(t) +& B(Y^y(t),Y^y(t))]dt = \sigma(Y^y(t)) dW(t)\\
&+\frac{\lambda_{N+1}}{2}P_N(X^x(t)-Y^y(t))dt + f(t)dt +dL^y(t).
\end{aligned}
$$

As an application of the main result Theorem \ref{ergodicity},
we have
\begin{theorem}
   Assume all the assumptions above are satisfied, then the reflected  stochastic Navier-Stokes equation (\ref{RNV}) ha s a unique invariant measure $\pi$. Further, there exists constants $C>0,r>0$ and $\tilde{N}>0$  such that
$$W_d(\mathbf{T_t}(x,\cdot),\pi)\leq C(1+V(x))e^{-rt},\quad t\geq 0,x\in H,$$
where $d(x,y)=\tilde{N}|x-y|_H^{\frac{2\delta}{1+\delta}}\wedge 1$.
\end{theorem}

\vskip 0.5cm

\noindent{\bf Acknowledgement}.

This work is partially supported by National Key R$\&$D program of China (No. 2022 YFA1006001)), National Natural Science Foundation of China (Nos. 12131019, 12371151, 12426655).

\noindent{\bf Data availability}.

No data was used for the search described in the article.

\noindent{\bf Disclosure statement}.

We declare that we have no conflict of interest.


\begin{thebibliography}{2}
\bibitem{BKS} O. Butkovsky, A. Kulik and M. Scheutzow, \textit{ Generalized coupling and ergodic rates for SPDEs and other Markov models.}  Ann. Appl. Probab., 2020, 30(1): 1-39.


\bibitem{BP} Z. Brzezniak and S. Peszat, \textit{Maximal inequalities and exponential estimates for stochastic convolutions in Banach spaces.} Stoch. Proces., Phys. \& Geom. I (Leipzig, 1999), 2000, 28: 55-64.

\bibitem{BZ} Z. Brzeźniak and T.S. Zhang, \textit{Reflection of stochastic evolution equations in infinite dimensional domains.} Ann. Inst. H.  Poincaré, 2023, 59(3): 1549-1571.


\bibitem{DP} C. Donati-Martin and E. Pardoux, \textit{White noise driven SPDEs with reflection.} Probab. Theory Relat. Fields, 1993, 95, 1-24.

\bibitem{LS} R.S. Liptser and A.N. Shiryayev \textit{Statistics of Random Processes.} I.Springer, Berlin, 2013.

\bibitem{DZ2}G. Da Prato and J. Zabczyk, \textit{J.: Stochastic Equations in Infinite Dimensions.} Cambridge University Press, Cambridge, 1992.

\bibitem{DZ} G. Da Prato and J. Zabczyk, \textit{Ergodicity for Infinite-Dimensional Systems.} London Mathematical Society Lecture Note Series 229, Cambridge University Press, Cambridge, UK, 1996.

\bibitem{FO} T. Funaki and S. Olla, \textit{Fluctuations for $\nabla\phi$ interface model on a wall.} Stoch. Proces. Appl., 2001, 94(1): 1-27.

\bibitem{GZ} N. Glatt-Holtz and M. Ziane, \textit{strong path wise solutions of the stochastic Navier-Stokes system}. Adv. Differential Equ., 2009, 14(5-6): 567-600.

\bibitem{HM} M. Hairer and J.C. Mattingly, \textit{Ergodicity of the 2D Navier–Stokes equations with degenerate stochastic forcing.} Ann. Math., 2006, 164(3): 993–1032.

\bibitem{HMS} M. Hairer, J.C. Mattingly and M. Scheutzow, \textit{Asymptotic coupling and a general form
 of Harris’ theorem with applications to stochastic delay equations.} Probab. Theory Relat. Fields, 2011, 149:
 223–259.

\bibitem{KPS} T. Komorowski, S. Peszat and T. Szarek, \textit{On ergodicity of some Markov processes.} Ann.  Probab., 2010, 38(4): 1401–1443.

\bibitem{K} A. Kulik, \textit{Ergodic Behavior of Markov Processes: With Applications to Limit Theorems.} De
Gruyter Studies in Mathematics 67.de Gruyter, Berlin, 2018.

\bibitem{KSS} R. Kapica, T. \'{S}zarek and M. Sleczka, \textit{On a unique ergodicity of some Markov processes.} Pot. Anal., 2012, 36: 589–606.


\bibitem{NP} D. Nualart and E. Pardoux, \textit{White noise driven quasilinear SPDEs with reflection.} Probab. Theory Relat. Fields, 1992, 93: 77-89.

\bibitem{PZ} S. Peszat and J. ZAbczyk, \textit{Strong Feller Property and Irreducibility for Diffusions on Hilbert Spaces.} Ann.  Probab., 1995, 23(1): 157-172.

\bibitem{V} C. Villani, \textit{Optimal Transport: Old and New.} Grundlehren der Mathematischen Wissenschaften
338.Springer, Berlin, 2009.

\bibitem{XZ} T. Xu and T.S. Zhang, \textit{White noise driven SPDEs with reflection: existence, uniqueness and large deviation principles.} Stochas. Process. Their Appl., 2009, 119(10): 3453-3470.

\bibitem{Z} X. Zhang, \textit{Exponential ergodicity of non-Lipschitz stochastic differential equations.} PAM. Math. Soc., 2009, 137(1): 329–337.
\end{thebibliography}
\end{document}